\documentclass[11pt, a4paper]{amsart}
\usepackage{amssymb,amsmath,amsthm}
\usepackage{hyperref}
\usepackage{graphicx}
\usepackage{xcolor}
\usepackage{graphicx} 
\newcommand{\Kon}{\mathcal{K}_o^n}
\newcommand{\latn}{\mathcal{L}^n}
\newcommand{\vol}{\mathrm{vol}\,}
\newcommand{\rg}{\mathrm{rg}\,}

\newcommand{\inte}{\mathrm{int}}

\newcommand{\lin}{\mathrm{lin}\,}
\newcommand{\R}{\mathbb{R}}
\newcommand{\Z}{\mathbb{Z}}
\newcommand{\N}{\mathbb{N}}

\newtheorem{theorem}{Theorem}[section]
\newtheorem{corollary}[theorem]{Corollary}
\newtheorem{lemma}[theorem]{Lemma}
\newtheorem{example}[theorem]{Example}%[section]
\newtheorem{remark}[theorem]{Remark}

\newtheorem{proposition}[theorem]{Proposition}  
\newtheorem{claim}{Claim}
\numberwithin{equation}{section}

\usepackage{pdfcomment}

\begin{document}

\title{Restricted Successive Minima} 
\author{Martin Henk}
\address{Fakult\"at f\"ur Mathematik, Otto-von-Guericke Universit\"at
  Magdeburg, Universit\"atsplatz 2, D-39106 Magdeburg, Germany}
\email{martin.henk@ovgu.de, carsten.thiel@ovgu.de}
\author{Carsten Thiel}

\keywords{}
\subjclass[2010]{52C07, 11H06}
\begin{abstract} We give bounds on the successive minima of an
  $o$-symmetric convex body under the restriction that the lattice
  points realizing the successive minima 
are not contained in a collection of forbidden sublattices.  Our
investigations extend former results to forbidden full-dimensional
lattices, to all successive minima and complement former results in
the lower dimensional case.
\end{abstract}
\maketitle

\section{Introduction}
Let $\Kon$ be the set of all $o$-symmetric convex bodies in $\R^n$ with non-empty interior, i.e., $K\in \Kon$ is an $n$-dimensional  compact convex set satisfying $K=-K$. The volume, i.e., the $n$-dimensional Lebesgue measure,   of a subset $X\subset\R^n$ is denoted by $\vol X$. 
By a lattice   $\Lambda\subset \R^n$ we understand a free $\Z$-module of rank $\rg\Lambda \leq n$. 
The set of all lattices is denoted by $\latn$,  and  $\det\Lambda$  denotes the determinant of $\Lambda\in\latn$, that is the $(\rg\Lambda)$-dimensional volume of a fundamental cell of $\Lambda$.  

For $K\in\Kon$ and $\Lambda\in\latn$, Minkowski introduced the $i$-th successive minimum $\lambda_i(K,\Lambda)$, $1\leq i \leq \rg\Lambda$, as the smallest positive number $\lambda$ such that $\lambda\,K$ contains at least $i$ linearly independent lattice points of $\Lambda$, i.e.,  
\begin{equation*}
 \lambda_i(K,\Lambda)=\min\{\lambda\in\R_{\geq 0}: \dim(\lambda\,K\cap\Lambda)\geq i\}, \quad 1\leq i\leq \rg\Lambda.
\end{equation*} 
Minkowski's first fundamental theorem, see e.g.~\cite[Sections 22--23]{Gruber:2007um}, on successive minima establishes an upper bound on the first successive minimum in terms of the volume of a convex body.
More precisely, for $K\in\Kon$ and $\Lambda\in\latn$ with $\rg\Lambda =r$ it may be formulated as 
\begin{equation}
\lambda_1(K,\Lambda)^r \vol_r(K\cap\lin\Lambda)\leq 2^r\det\Lambda,
\label{eq:minkowski_first}
\end{equation}  
where $\vol_r(\cdot)$ denotes the $r$-dimensional
volume, here with
respect to the subspace $\lin\Lambda$, the linear hull of
$\Lambda$. In the case $r=n$ we just write $\vol(\cdot)$. One of the many successful applications of this inequality is 
related to the so-called ``Siegel's lemma'', which is a synonym for 
results bounding the norm of a non-trivial  lattice point lying in a linear
subspace  given as $\ker A$ where $A\in\Z^{m\times n}$ is
an integral matrix of rank $m$. For instance, with respect to the
maximum norm $|\cdot|_\infty$, it was shown by Bombieri\&Vaaler
\cite{Bombieri:1983uz} (see also Ball\&Pajor \cite{Ball1990}) that there exists a
$z\in\ker A\setminus\{0\}$ such that 
\begin{equation*}
  |z|_\infty \leq  \sqrt{\det(A\,A^\intercal)}^\frac{1}{n-m}. 
\end{equation*}
In fact, this follows by \eqref{eq:minkowski_first}, where $K=[-1,1]^n$
is the cube of edge length $2$, $\Lambda=\ker
A\cap\Z^n$ is an $(n-m)$ dimensional lattice of  determinant $\leq \sqrt{\det(A\,A^\intercal)}$, and Vaaler's
result on the minimal volume of  a slice  of a cube (cf. 
~\cite{Vaaler1979}) which here gives 
$\vol_{n-m}([-1,1]^n\cap\lin\Lambda)\geq 2^{n-m}$. For generalizations 
of Siegel's lemma to number fields we refer to
\cite{Bombieri:1983uz}, 
\cite{Fukshansky:2006ti},\cite{Fukshansky:2006db}, \cite{Gaudron:2009iu}, \cite{Gaudron:2012hq}, \cite{Vaaler:2003dw} and the references within.

Motivated by questions in Diophantine approximation, Fukshansky
studied in \cite{Fukshansky:2006ti} an inverse problem to that addressed in
Siegel's lemma, namely to bound the norm of lattice points which are
not contained in the union of proper sublattices. In order to describe
his result we need a bit more notation. 

For a collection of sublattices $\Lambda_i\subset\Lambda$, $1\leq i\leq
s$,   with $\cup_{i=1}^s\Lambda_i\ne \Lambda$ we call 
 \begin{equation*}
 \lambda_i(K,\Lambda\setminus\cup_{i=1}^s\Lambda_i)=\min\{\lambda\in\R_{\geq 0}: \dim(\lambda\,K\cap \Lambda\setminus\cup_{i=1}^s\Lambda_i)\geq i\}, \quad 1\leq i\leq \rg\Lambda,
\end{equation*} 
the {\em $i$-th restricted successive minimum} of $K$ with respect to
$\Lambda\backslash\cup_{i=1}^s\Lambda_i$.  Observe that by the compactness of $K$
and the discreteness of $\Lambda\backslash\cup_{i=1}^s\Lambda_i$ these
minima are well-defined.
Furthermore, they behave nicely
 with respect to dilations, as for $\mu>0$ we have 
\begin{equation}
\lambda_i(\mu\,K,\Lambda\setminus\cup_{i=1}^s\Lambda_i)=
\lambda_i(K,\tfrac{1}{\mu}(\Lambda\setminus\cup_{i=1}^s\Lambda_i))=\tfrac{1}{\mu}\,\lambda_i(K,\Lambda\setminus\cup_{i=1}^s\Lambda_i).
\label{eq:dilation}
\end{equation}

Moreover, for a lattice
$\Lambda\in\latn$, $r=\rg\Lambda$, and a basis $(b_1,\dots,b_r)$,
$b_j\in\R^n$, of $\Lambda$ let $v(\Lambda)\in\R^{\binom{n}{r}}$ be the
vector with entries  $\det B_j$, where $B_j$ is an $r\times r$
submatrix of $(b_1,\dots,b_r)$. Observe that up to the order of the
coordinates the vector is independent of the given basis, and on
account of the Cauchy-Binet formula the
Euclidean norm of  $v(\Lambda)$ is the determinant of the lattice.
With this notation, Fukshansky proved \cite[Theorem
1.1]{Fukshansky:2006ti}
\begin{equation}
\lambda_1\left([-1,1]^n,
  \Lambda\setminus\cup_{i=1}^s\Lambda_i\right)\leq
\left(\frac{3}{2}\right)^{r-1}r^r\left(\sum_{i=1}^s\frac{1}{|v(\Lambda_i)|_\infty}+\sqrt{s}\right)|v(\Lambda)|_\infty
+1, 
\label{eq:fukshansky}
\end{equation}
for proper sublattices $\Lambda_i$, i.e.,  $\rg\Lambda_i<\rg\Lambda=r$,
$1\leq i\leq s$. This result was generalized and improved in various
ways by Gaudron \cite{Gaudron:2009iu} and Gaudron\&R{\'emond}
\cite{Gaudron:2012hq}. In particular, \eqref{eq:fukshansky} has been
extended to all  $o$-symmetric bodies
as well as  to the adelic setting (see also \cite[Lemma
3.2]{Gaudron:tk} for an application). For instance, the following is a 
simplified version of  \cite[Theorem 6.1]{Gaudron:2009iu} when we assume that
$\rg\Lambda_i=\rg\Lambda-1=r-1$ (see also
\cite[Theorem 2.2, Corollary 3.3]{Gaudron:2012hq})
\begin{equation} 
\begin{split}
\lambda_1\left(K,
  \Lambda\setminus\cup_{i=1}^s\Lambda_i\right) \leq  \nu\,\max_{1\leq
  i\leq
  s}\Bigg\{1,&\frac{\nu^{r-1}\vol(K\cap\lin\Lambda_i)}{\omega_r\,\det\Lambda_i},\\
&\left(\frac{\nu}{\lambda_1(K,\Lambda\cap\lin\Lambda_i)}\right)^{\frac{r-2}{2}}\Bigg\}, 
\end{split}
\label{eq:gaudron}
\end{equation}
where $\nu=7\,r(s\,\omega_r\det\Lambda/\vol(K))^{1/r}$ and  $\omega_r$ is
the volume of the $r$-dimensional unit ball. For an algorithmic
treatment of $\lambda_1\left(K,
  \Lambda\setminus\cup_{i=1}^s\Lambda_i\right)$
in the case $s=1$ and $K$ the unit ball of an $l_p$-norm we refer to 
\cite{Naewe:2007vl}.

In our first theorem we want to complement  these results on
forbidden lower dimensional lattices  by a bound 
which takes care of the size or the structure of the
individual forbidden sublattices such that the bound becomes essentially
\eqref{eq:minkowski_first} if $\lambda_1(K,\Lambda_i)\to\infty$ for
$1\leq i\leq s$. In this case %  and under the 
% additional assumption that $\det\Lambda_i\to\infty$, $1\leq
% i\leq s$,
the bounds  in \eqref{eq:fukshansky} and \eqref{eq:gaudron}
still have a dependency on $s$ of order $\sqrt{s}$ and
$s^{1/r}$, respectively. 
Here we have the following result. 
\begin{theorem} Let $K\in\Kon$, $\Lambda\in\latn$, $\rg\Lambda=n\geq 2$, and
  let $\Lambda_i\subset\Lambda$, $1\leq i\leq s$, $\rg\Lambda_i\leq
  n-1$, be sublattices. Then 
\begin{equation*}
 \lambda_1\left(K,
  \Lambda\setminus\cup_{i=1}^s\Lambda_i\right)<  6^{n-1}\frac{\det\Lambda}{\lambda_1(K,\Lambda)^{n-2}\vol(K)}\left(\sum_{i=1}^s\frac{1}{\lambda_1(K,\Lambda_i)}\right)+\sqrt[n]{2^n\frac{\det\Lambda}{\vol(K)}}.
\end{equation*}
\label{thm:lower}
\end{theorem}
Observe that, in case $s=0$ or all $\lambda_1(K,\Lambda_i)$ very large, we
get essentially \eqref{eq:minkowski_first}. 
Our second main theorem deals with  forbidden full-dimensional sublattices,
i.e., $\rg\Lambda_i=\rg\Lambda$, $1\leq i\leq s$, 
and here we will show.
\begin{theorem} Let $K\in\Kon$, $\Lambda\in\latn$, $\rg\Lambda=n\geq 2$, and
  let $\Lambda_i\subset\Lambda$, $1\leq i\leq s$, $\rg\Lambda_i=
  n$, be sublattices such that $\cup_{i=1}^s\Lambda_i\ne\Lambda$. Then 
\begin{equation*}
 \lambda_1\left(K,
  \Lambda\setminus\cup_{i=1}^s\Lambda_i\right)< \frac{2^n\det\Lambda}{\lambda_1(K,\overline{\Lambda})^{n-1}\vol(K)}
\biggl(\sum_{i=1}^s \frac{\det\overline{\Lambda}}{\det\Lambda_i} -s+1
\biggr)+
\lambda_1(K,\overline{\Lambda}),
\end{equation*}
where $\overline{\Lambda}=\bigcap_{i=1}^s\Lambda_i$.
\label{thm:full}
\end{theorem}

In the special case $s=1$ the theorem above can be formulated 
as 
 \begin{corollary} Let $K\in\Kon$, $\Lambda\in\latn$, $\rg\Lambda=n\geq 2$, and
  let $\Lambda_1\subsetneq\Lambda$, $\rg\Lambda_1= n$, be a sublattice. Then 
\begin{equation*}
 \lambda_1\left(K,
  \Lambda\setminus\Lambda_1\right)\leq \frac{2^n\det\Lambda}{\lambda_1(K,\Lambda_1)^{n-1}\vol(K)}+
\lambda_1(K,\Lambda).
\end{equation*}
\label{cor:onefull}
\end{corollary}
Indeed, the corollary is just an immediate consequence of Theorem
\ref{thm:full}, since in this case we 
 may
assume
$\lambda_1(K,\Lambda_1)=\lambda_1(K,\Lambda)$.

The following
example shows that the bound in Theorem~\ref{thm:full} as well as the
one of the corollary above cannot be improved in general by a
multiplicative factor. 
\begin{example} Let $K\in\mathcal{K}^2_o$ be the rectangle
  $K=[-1,1]\times[-\alpha,\alpha]$ of edge-lengths 2 and
  $2\,\alpha$, $\alpha\leq 1$, and of volume $4\alpha$. 
Let $\Lambda=\Z^2$, and let $\Lambda_1,\Lambda_2\subset\Z^2$ be
  the sublattices 
\begin{equation*}
\Lambda_1=\{(z_1,z_2)^\intercal\in\Z^2: z_2\equiv
  0\bmod 2\},\, \Lambda_2=\{(z_1,z_2)^\intercal\in\Z^2: z_1\equiv
  0\bmod p\},
\end{equation*}
where $p>2$ is a prime.
Then $\det\Lambda=1$, $\det\Lambda_1=2$, $\det\Lambda_2=p$,
 and 
\begin{equation*}
\overline\Lambda=\Lambda_1\cap\Lambda_2 =\{(z_1,z_2)^\intercal\in\Z^2: z_2\equiv
  0\bmod 2,\,z_1\equiv
  0\bmod p\}
\end{equation*}
with $\det\overline\Lambda=2p$. For $\alpha \leq 2/p$
we therefore have $\lambda_1(K,\overline{\Lambda})=p$. Regarding
the set $\Lambda\setminus(\Lambda_1\cup\Lambda_2)$ we
observe that the lattice points on the axes are forbidden, but not
$(1,1)^\intercal$ and so
$\lambda_1(K,\Lambda\setminus(\Lambda_1\cup\Lambda_2))=1/\alpha$. Putting
everything together, the bound in Theorem \ref{thm:full} evaluates for
$\alpha\leq 2/p$ to
\begin{equation*}
\frac{1}{\alpha}=\lambda_1(\Lambda\setminus(\Lambda_1\cup\Lambda_2))
< \frac{4}{p\,4\alpha}(p+1) + p = \frac{1}{\alpha}+\frac{1}{p\alpha}+p.
\end{equation*}
Hence for $\alpha=2/p^2$ and  $p\to \infty$ the bound cannot be
improved by a multiplicative factor.

In the situation of Corollary \ref{cor:onefull}, i.e., we consider only the
forbidden lattice $\Lambda_1$, the upper bound in the corollary evaluates to
$\frac{1}{\alpha}+1$, whereas, as before,
$\lambda_1(K,\Lambda\setminus\Lambda_1)=1/\alpha$.
\end{example}
Before
beginning with the proofs of our results we would like to mention 
a closely related problem, namely  to cover $K\cap\Lambda$, $K\in\Kon$
by a minimal number
$\gamma(K)$ of lattice hyperplanes.  Obviously, having a $\nu>0$ with
$\gamma(\nu K)\geq s+1$ implies that 
\begin{equation*}
 \lambda_1(K,\Lambda\setminus\cup_{i=1}^s\Lambda_i)\leq \nu
\end{equation*}
in the case of lower dimensional sublattices $\Lambda_i$.
For bounds on  $\gamma(K)$ in terms of the successive minima and other
functionals from the Geometry of Numbers we refer to   
 \cite{Barany:2001uz}, \cite{Bezdek:1994wt}, \cite{Bezdek:2009bc}.

The proof of Theorem
\ref{thm:lower} will be given in the next section and
full-dimensional forbidden sublattices, i.e., 
Theorem \ref{thm:full}, will be treated in  Section
\ref{sec:full}. In each of the sections we also present some extensions of the
 results above to higher successive minima, i.e., to  $\lambda_i\left(K,
   \Lambda\setminus\cup_{i=1}^s\Lambda_i\right)$, $i>1$.

\section{Avoiding lower-dimensional sublattices}\label{sec:lower}
In the course of the proof we have to estimate the number of
lattice points in a centrally symmetric convex body, i.e., to bound
$|K\cap\Lambda|$ from below and above. Assuming $K\in\Kon$ and
$\rg\Lambda=n$, we will
use as a lower bound a  classical result of van der Corput, see, e.g.,
\cite[p. 51]{Gruber:1987vp}, 
\begin{equation}
|K\cap \Lambda|\geq 2\left\lfloor\frac{\vol(K)}{2^n\det\Lambda}\right\rfloor +1 > \frac{\vol(K)}{2^{n-1}\det\Lambda}-1. 
\label{eq:vandercorput}
\end{equation}
As upper bound we will use a bound in terms of the first successive
minima by Betke et al. \cite{Betke1993e} 
\begin{equation}
|K\cap \Lambda| \leq\left(\frac{2}{\lambda_1(K,\Lambda)}+1\right)^{n}.
\label{eq:first_lattice}
\end{equation}

\begin{proof}[Proof of Theorem \ref{thm:lower}] By scaling $K$ with
  $\lambda_1(K,\Lambda)$ we may assume without loss of 
  generality that $\lambda_1(K,\Lambda)=1$, i.e., $K$ contains no 
  non-trivial lattice point in its interior (cf.~\eqref{eq:dilation}). Let
  $n_i=\rg\Lambda_i<n$. % , $1\leq i\leq s$, 
  By  \eqref{eq:first_lattice} we get  for $\gamma\geq 1$ and on
  account of $\lambda_1(K,\Lambda_i)\geq\lambda_1(K,\Lambda)=1$ we get 
\begin{equation}
  |\gamma\,K\setminus\{0\}\cap \Lambda_i |\leq
  \left(\gamma\,\frac{2}{\lambda_1(K,\Lambda_i)}+1\right)^{n_i} 
  -1<\gamma^{n-1} 3^{n-1}\frac{1}{\lambda_1(K,\Lambda_i)}.
\label{eq:ineq_subspaces}
\end{equation}
Hence, for $\gamma\geq 1$ we have 
\begin{equation}
|\gamma\,K\setminus\{0\}\cap (\cup_{i=1}^s \Lambda_i)| < \gamma^{n-1}\,3^{n-1}\,\sum_{i=1}^s
 \frac{1}{\lambda_1(K,\Lambda_i)}.
\end{equation}
Combining this bound with the upper bound \eqref{eq:vandercorput}
leads for $\gamma\geq 1$ to the estimate 
\begin{equation*}
\begin{split}
|\gamma K\setminus\{0\}&\cap \Lambda\setminus \cup_{i=1}^s \Lambda_i| >
\gamma^n\frac{\vol(K)}{2^{n-1}\det\Lambda}-2 - |\gamma\,K\setminus\{0\}\cap
(\cup_{i=1}^s \Lambda_i)|\\
 &> \gamma^n\frac{\vol(K)}{2^{n-1}\det\Lambda}-\gamma^{n-1}\,3^{n-1}\left(\sum_{i=1}^s
 \frac{1}{\lambda_1(K,\Lambda_i)}\right) -2\\ 
&=\frac{\vol(K)}{2^{n-1}\det\Lambda}\left(\gamma^n-\gamma^{n-1}\,\beta-\rho\right),
\end{split}
\end{equation*}
where 
\begin{equation*}
 \beta=6^{n-1}\frac{\det\Lambda}{\vol(K)}\left(\sum_{i=1}^s
 \frac{1}{\lambda_1(K,\Lambda_i)}\right),\quad
\rho=2^n\frac{\det\Lambda}{\vol(K)}.
\end{equation*}
Hence, depending on $\beta$ and $\rho$, we have to determine a
$\gamma\geq 1$ such that $\gamma^n-\gamma^{n-1}\beta -\rho >0$. To
this end let $\overline{\gamma}=\beta+\rho^{1/n}$. Then 
\begin{equation}
\begin{split}
\overline{\gamma}^n-\overline{\gamma}^{n-1}\beta & =
(\beta+\rho^{1/n})^n-(\beta+\rho^{1/n})^{n-1}\beta
\\ &=\rho^{1/n}(\beta+\rho^{1/n})^{n-1} 
 >\rho^{1/n}\rho^{(n-1)/n}=\rho.
\end{split}
\label{eq:strict}
\end{equation}
Finally, we observe that
\begin{equation*}
\overline{\gamma}>\rho^{1/n}=(2^{n}\det\Lambda/\vol(K))^{1/n}\geq
\lambda_1(K,\Lambda)=1
\end{equation*}
 by \eqref{eq:minkowski_first} and our
assumption. Hence, $\overline{\gamma}>1$ and in view of
\eqref{eq:strict} we have $\lambda_1(K,\Lambda\setminus\cup_{i=1}^s\Lambda_i)< 
\overline{\gamma}$ which by the definition of $\overline{\gamma}$
yields the desired bound of the theorem with respect to our normalization 
 $\lambda_1(K,\Lambda)=1$. 
\end{proof}
Compared to the bounds in \eqref{eq:fukshansky} and
\eqref{eq:gaudron},   our
formula uses only the successive minima and not the determinants of the
forbidden sublattices which reflect more the structure of a lattice. 
However, instead of  \eqref{eq:first_lattice} one
can also use a recent Blichfeldt-type bound for $o$-symmetric convex
 bodies $K$ with $\dim(K\cap\Lambda)=n$  due to Henze \cite{Henze:2012tr} 
  \begin{equation*}
  |K\cap \Lambda|\leq
  \frac{n!}{2^n}\frac{\vol(K)}{\det\Lambda}L_n(-2),  
  \label{eq:henze}
  \end{equation*}
  where $L_n(x)$ is the $n$-th Laguerre polynomial. This leads
  to a bound on $\lambda_1(K,\Lambda\setminus\cup_{i=1}^s\Lambda_i)$
  where  the sum over $1/\lambda_1(K,\Lambda_i)$ is
   replaced by a sum over ratios of the type $\vol_{\dim H}(K\cap
   H)/\det(\Lambda_i\cap H)$ for certain lower dimensional planes
   $H\subseteq\lin\Lambda_i$. In general, however, we have no control
   over the dimension of these hyperplanes $H$ nor on the volume of the sections.

Theorem \ref{thm:lower} can easily be extended inductively to higher restricted
successive minima
$\lambda_{j+1}\left(K,\Lambda\setminus\cup_{i=1}^s\Lambda_i\right)$, $1 \leq j \leq n-1$, by
avoiding, in addition, a $j$-dimensional lattice containing $j$ linearly
independent lattice points corresponding to the successive minima 
$\lambda_{i}\left(K,\Lambda\setminus\cup_{i=1}^s\Lambda_i\right)$,
$1\leq i\leq j$. 
\begin{corollary} Under the assumptions of Theorem \ref{thm:lower} we
  have for $j=1\dots,n-1$ 
\begin{equation*}
\begin{split}
 \lambda_{j+1}\left(K,
  \Lambda\setminus\cup_{i=1}^s\Lambda_i\right)&<
6^{n-1}\frac{\det\Lambda}{\lambda_1(K,\Lambda)^{n-2}\vol(K)}\left(\sum_{i=1}^s\frac{1}{\lambda_1(K,\Lambda_i)}\right)\\
&+\left( \frac{3^j}{\lambda_1(K,\Lambda)^j}2^{n-1} \frac{\det\Lambda}{\vol(K)}+\left(2^n \frac{\det\Lambda}{\vol(K)}\right)^{\frac{n-j}{n}}\right)^\frac{1}{n-j}.
\end{split}
\end{equation*}
\end{corollary}         
\begin{proof} Let $z_i\in \lambda_{i}\left(K,
  \Lambda\setminus\cup_{i=1}^s\Lambda_i\right)\, K\cap\Lambda$, $1\leq
i\leq j$, be linearly independent, and let
$\overline{\Lambda}\subset\Lambda$ be the $j$-dimensional lattice
generated by these vectors. Then 
\begin{equation}
 \lambda_{j+1}\left(K,
  \Lambda\setminus\cup_{i=1}^s\Lambda_i\right)=\lambda_{1}\left(K,
  \Lambda\setminus(\cup_{i=1}^s\Lambda_i\cup \overline{\Lambda})\right)
\label{eq:succ_sub}
\end{equation}
and we now follow the proof of Theorem \ref{thm:lower}. In
particular, we assume $\lambda_1(K,\Lambda)=1$.  
In addition to the upper bounds on $|\gamma\,K\setminus\{0\}\cap \Lambda_i|$,
$1\leq i\leq s$, in \eqref{eq:ineq_subspaces}, we also use for
$\gamma\geq \lambda_1(K,\overline{\Lambda})
\geq\lambda_1(K,\Lambda)=1$  the bound 
\begin{equation}
|\gamma K\setminus\{0\}\cap\overline{\Lambda}|<
\left(\frac{2\,\gamma}{\lambda_1(K,\overline{\Lambda})}+1\right)^j\leq
3^j \left(\frac{\gamma}{\lambda_1(K,\overline{\Lambda})}\right)^j.
\end{equation}
Combining this bound with  \eqref{eq:vandercorput}
leads for $\gamma\geq \lambda_1(K,\overline{\Lambda})$ to 
\begin{equation*}
\begin{split}
|\gamma K\setminus\{0\}&\cap \Lambda\setminus (\cup_{i=1}^s \Lambda_i\cup\overline{\Lambda})| >
\gamma^n\frac{\vol(K)}{2^{n-1}\det\Lambda}-2 \\ &\hphantom{hugo} - |\gamma\,K\setminus\{0\}\cap
(\cup_{i=1}^s \Lambda_i)|-|\gamma K\setminus\{0\}\cap\overline{\Lambda}|\\
 &> \gamma^n\frac{\vol(K)}{2^{n-1}\det\Lambda} -2 -\gamma^{n-1}\,3^{n-1}\left(\sum_{i=1}^s
 \frac{1}{\lambda_1(K,\Lambda_i)}\right)-3^j \left(\frac{\gamma}{\lambda_1(K,\overline{\Lambda})}\right)^j\\ 
&=\frac{\vol(K)}{2^{n-1}\det\Lambda}\left(\gamma^n-\gamma^{n-1}\,\beta-\gamma^j\alpha-\rho\right),
\end{split}
\end{equation*}
with 
\begin{equation*}
\begin{split}
\beta &= 6^{n-1}\frac{\det\Lambda}{\vol(K)}\left(\sum_{i=1}^s\frac{1}{\lambda_1(K,\Lambda_i)}\right),\quad
\alpha =
\frac{3^j}{\lambda_1(K,\overline{\Lambda})^j}2^{n-1}\frac{\det\Lambda}{\vol(K)},\\ 
\rho &=2^n\frac{\det\Lambda}{\vol(K)}.
\end{split}
\end{equation*}
Setting now
$\overline{\gamma}=\beta+(\alpha+\rho^\frac{n-j}{n})^\frac{1}{n-j}$ we
see as in the proof of Theorem \ref{thm:lower} that 
\begin{equation} 
\begin{split}
  \overline{\gamma}^n-\overline{\gamma}^{n-1}\,\beta-\overline{\gamma}^j\alpha-\rho
  &=\overline{\gamma}^j(\overline{\gamma}^{n-j}-\beta\overline{\gamma}^{n-j-1}-\alpha)-\rho\\ &>
  \overline{\gamma}^j\rho^\frac{n-j}{n}-\rho > 0.
\end{split} 
\label{eq:strict_cor}
\end{equation}
Since $\overline{\gamma}>\beta +\rho^\frac{1}{n}$ which is by the proof of
Theorem \ref{thm:lower} an
upper bound on $\lambda_1(K,\overline{\Lambda})$ we also have 
$\overline{\gamma}>\lambda_1(K,\overline{\Lambda})$ and so we know 
$\lambda_{j+1}(K,\Lambda\setminus\{\cup_{i=1}^s\Lambda_i\})<\overline{\gamma}$
(cf.~\eqref{eq:strict_cor}).  By the definition of $\overline{\gamma}$
we get the required upper bound with respect to the normalization $\lambda_1(K,\Lambda)=1$.
\end{proof}

An upper bound on $\lambda_j(K, \Lambda\setminus\{\cup_{i=1}^s\Lambda_i\})$ of a different kind involves the so-called covering radius 
$\mu(K,\Lambda)$ of a
convex body $K\in\Kon$ and a lattice $\Lambda\in\latn$,
$\rg\Lambda=n$. It is the smallest positive number $\mu$ such that
any translate of $\mu\,K$ contains a lattice point,
i.e. (cf.~\cite[Sec.13, Ch.2]{Gruber:1987vp}), 
\begin{equation*}
 \mu(K,\Lambda)=\min\{\mu>0 : (t+\mu K)\cap\Lambda\ne\emptyset\text{
   for all }t\in\R^n\}.
\end{equation*}

\begin{proposition} Under the assumptions of Theorem \ref{thm:lower}
  we have 
\begin{equation*}
\lambda_{1}(K,\Lambda\setminus(\cup_{i=1}^s\Lambda_i))\leq
(s+1)\,\mu(K,\Lambda) 
\end{equation*}
and hence, $\lambda_{j}(K,\Lambda\setminus(\cup_{i=1}^s\Lambda_i))\leq
(s+2)\,\mu(K,\Lambda)$ for $2\leq j\leq n$.
\end{proposition} 
\begin{proof} Observe that on account of \eqref{eq:succ_sub} the bound
  for $j\geq 2$ follows from the one for
  $\lambda_{1}(K,\Lambda\setminus\{\cup_{i=1}^s\Lambda_i\})$. For 
  the proof in case $j=1$ let 
  $H_i=\lin\Lambda_i$, $1\leq i\leq s$, and for short we write
  $\overline{\mu}$ instead of $\mu(K,\Lambda)$. By
  Ball's \cite{Ball:1991wn} solution of the affine plank problem for $o$-symmetric convex
  bodies, applied to $\overline{\mu}\,K$, we know that there
  exists a $t\in\R^n$ such that 
\begin{equation*}
   (t+\frac{1}{s+1}\overline{\mu}\,K) \subset \overline{\mu}\,K \quad\text{ and }\quad
   \inte(t+\frac{1}{s+1}\overline{\mu}\,K)\cap  H_i=\emptyset, \,1\leq i\leq s,
\end{equation*}
where $\inte(\cdot)$ denotes the interior. Thus, for any $\epsilon>0$ the
body  $(s+1+\epsilon)\overline{\mu}\,K$
contains a translate $t_\epsilon+\overline{\mu}\,K$ having no  points
in common with $H_i$, $1\leq i\leq s$. Hence, together with the definition of the
covering radius we have $(t_\epsilon+\overline{\mu}K)\cap
\Lambda\setminus(\cup_{i=1}^s\Lambda_i)\ne\emptyset$   and so
$\lambda_{1}(K,\Lambda\setminus(\cup_{i=1}^s\Lambda_i))\leq
(s+1+\epsilon)\overline{\mu}$. By the arbitrariness of $\epsilon$ and the compactness of $K$ the
assertion follows.
\end{proof}
For a comparable uniform bound in the much more general adelic setting
and, of course, with a completely different method see \cite[Proposition 3.2]{Gaudron:2012hq}. 

\section{Avoiding full-dimensional sublattices}\label{sec:full}
If the forbidden sublattices are full-dimensional we  cannot argue 
as in the lower-dimensional case, since now the  number
of forbidden lattice points in $\lambda K\cap(\cup_{i=1}^s\Lambda_i)$
grows with  the same order of magnitude with
respect to $\lambda$ as the number of points in $\lambda\,K\cap\Lambda$. 

The tool
we are using in this full-dimensional case is the torus group  
$\R^n/\overline{\Lambda}$ for a certain lattice
$\overline{\Lambda}$. For a more detailed discussion we refer to \cite[Section 26]{Gruber:2007um}.
We recall that this quotient of abelian groups is
a compact topological group and we may identify
$\R^n/\overline{\Lambda}$ with a fundamental parallelepiped  $P$ of
$\overline{\Lambda}$, i.e., 
\begin{equation*}
  \R^n/\overline{\Lambda}\sim P=\{\rho_1\,b_1+\cdots +\rho_n\,b_n : 0\leq
  \rho_i <1\},
\end{equation*} 
where $b_1,\dots,b_n$ form a basis of $\overline{\Lambda}$. Hence
for $X\subset\R^n$, the set $X$ modulo $\overline{\Lambda}$, i.e.,
$X/\overline{\Lambda}$, can be described (thought of) as 
\begin{equation*}
  X/\overline{\Lambda}=\{y\in P: \exists\, b\in\overline{\Lambda} \text{ s.t. }
 y+b\in X\}=(\overline{\Lambda}+X)\cap P
\end{equation*}
and we can think of $\overline{X}\subseteq\R^n/\overline{\Lambda}$ as its image under inclusion into $\R^n$.
In the same spirit we may identify addition $\oplus$ 
  in $\R^n/\overline{\Lambda}$ with the corresponding operation
  in $\R^n$, i.e., for  $\overline{X}_1,
\overline{X}_2\subset\R^n/\overline{\Lambda}$ we
have 
\begin{equation*}
  \overline{X}_1\oplus \overline{X}_2 = ((\overline{X}_1+
\overline{X}_2)+\overline{\Lambda})\cap P.
\end{equation*}
As $\R^n/\overline{\Lambda}$ is a compact
abelian group, there is a unique Haar measure
$\vol_{T}(\cdot)$ on it  normalized to
$\vol_{T}(\R^n/\overline{\Lambda})=\det\overline{\Lambda}$, and for
a ``nice'' measurable set $X\subset \R^n$ or $\overline{X}\subset
\R^n/\overline{\Lambda}$ we have 
\begin{equation*} 
  \vol_T(X/\overline{\Lambda})=\vol((\overline{\Lambda}+X)\cap P)
  \text{ and } \vol_T(\overline{X})=\vol((\overline{\Lambda}+\overline{X})\cap P).
\end{equation*} 
Regarding the volume of the sum of two sets  $\overline{X}_1, \overline{X}_2\subset\R^n/\overline{\Lambda}$
we have the following classical so-called sum theorem of 
Kneser and Macbeath \cite[Theorem 26.1]{Gruber:2007um}
\begin{equation}
 \vol_T(\overline{X}_1\oplus \overline{X}_2)\geq
 \min\{\vol_T(\overline{X}_1)+\vol_T(\overline{X}_2),\det\overline{\Lambda}\}
.
\label{eq:torus_add}
\end{equation}
We also note that for an $o$-symmetric convex body $K\in\Kon$ and
$\lambda\geq 0$ the set $\overline{\Lambda}+\lambda K$ forms a lattice
packing, i.e., for any two different lattice points
$\overline{a},\overline{b}\in\overline{\Lambda}$ the translates
$\overline{a}+\lambda K$ and $\overline{b}+\lambda K$ do not
overlap, if and only if
$\lambda\leq\lambda_1(K,\overline{\Lambda})/2$. Hence we know for $0\leq \lambda\leq\lambda_1(K,\overline{\Lambda})/2$
\begin{equation}
 \vol_T((\lambda\,K)/\overline{\Lambda}))=\vol((\lambda
 K+\overline{\Lambda})\cap P) =\lambda^n \vol(K).
\label{eq:torus_scal}
\end{equation}
Furthermore, we also need a ``torus version'' of  van der
Corput's result \eqref{eq:vandercorput}  

\begin{lemma} Let $K\in\Kon$, $\Lambda\in\latn$, $\rg\Lambda=n$ and let
  $\overline\Lambda\subsetneq\Lambda$  be a sublattice 
  with $\rg\overline\Lambda=n$, and let $m\in\N$ with $m\det\Lambda<\det\overline\Lambda$. 
  If $\vol_T(\frac{1}{2}K/\overline\Lambda)\geq m\det\Lambda$ then 
\begin{equation*}
      \#\left(K/\overline\Lambda\cap\Lambda\right) \geq m+1, 
\end{equation*}
i.e., $K$ contains at least $m+1$ lattice points of $\Lambda$
belonging to different cosets modulo $\overline\Lambda$.
\label{lem:vandercorputtorus}
\end{lemma}
\begin{proof}  By the compactness of $K$ and the discreteness of
   lattices we may assume $\vol_T(\frac{1}{2}K/\overline\Lambda)> m\det\Lambda$.
  Let $P$ be a fundamental parallelepiped of the lattice
  $\overline\Lambda$. Then by assumption we have for the measurable
  set $X=(\frac{1}{2}K+\overline\Lambda)\cap P$ that $\vol X>
  m\det\Lambda$. According to a result of van der Corput \cite[Theorem
  1, Sec 6.1]{Gruber:1987vp} we know that there exists pairwise
  different $x_i\in X$, $1\leq i\leq m+1$, such that $x_i-x_j\in
  \Lambda$. By the $o$-symmetry and convexity of $K$ we have
  $(X-X)=(K+\overline\Lambda)\cap (P-P)$  and since $(P-P)\cap\overline\Lambda=\{0\}$ we conclude  
\begin{equation*}
      x_i-x_j\in (K+\overline\Lambda)\cap \Lambda\setminus\overline\Lambda,\quad i\ne j.  
\end{equation*}
Hence the $m$ points $x_i-x_1\in K+\overline\Lambda$, $i=2,\dots,m+1$,
belong to different non-trivial cosets of $\Lambda$ modulo $\overline\Lambda$ and
thus $\#\left(K/\overline\Lambda\cap\Lambda\right) \geq m+1$, where
the additional $1$ counts the origin.
\end{proof}

The next lemma states some simple facts on the intersection of
full-dimensional sublattices. 
\begin{lemma}  Let $\Lambda\in\latn$, $\Lambda_i\subseteq\Lambda$,
  $1\leq i\leq s$, $\rg\Lambda_i=\rg\Lambda=n$, and  let
  $\overline{\Lambda}=\cap_{i=1}^n\Lambda_i$. Then $\overline{\Lambda}\in\latn$  with $\rg\overline\Lambda=n$, and   
\begin{equation*}
\max_{1\leq i \leq
  s}\det\Lambda_i\leq\det\overline{\Lambda}\leq(\det\Lambda)^{1-s}\det\Lambda_1\cdot\ldots\cdot\det\Lambda_s.
\end{equation*}
Moreover, with
$m=\sum_{i=1}^s\frac{\det\overline\Lambda}{\det\Lambda_i}-s+1$ we have 
\begin{enumerate}
\item[\rm i)] The union $\cup_{i=1}^s\Lambda_i$ is covered by at most $m$
  cosets of $\Lambda$ modulo $\overline\Lambda$. 
\item[{\rm ii)}] If $\frac{\det\overline{\Lambda}}{\det\Lambda} \geq m+1$ then  $\Lambda \neq \bigcup_{i=1}^s \Lambda_i$.
\end{enumerate}
\label{lem:lattice_cap}
\end{lemma} 
\begin{proof} In order to show that $\overline\Lambda$ is a
  full-dimensional lattice  it suffices to consider $s=2$. Obviously,
  $\Lambda_1\cap\Lambda_2$ is a discrete subgroup  of $\Lambda$ and it
  also 
  contais $n$ linearly independent points, e.g.,
  $(\det\Lambda_2)a_1,\dots,  (\det\Lambda_2)a_n$, where
  $a_1,\dots,a_n$ is a basis of $\Lambda_1$. Hence $\overline\Lambda$ is
  a full-dimensional lattice, cf.~\cite[Theorem 2, Sec 3.2]{Gruber:1987vp}.   The lower
  bound on $\det\overline\Lambda$  is clear by the inclusion
  $\overline\Lambda\subseteq\Lambda_i$, $1\leq i\leq s$. For the
  upper bound we observe that two points $g,h \in\Lambda$ belong to
  different cosets modulo $\overline\Lambda$ if and only if $g$ and $h$ belong
  to different cosets of $\Lambda$ modulo at least one $\Lambda_i$. 
   There are $\det\Lambda_i/\det\Lambda$ many cosets for each $i$ and so we get the upper bound.  

For i) we note that    $\Lambda_i$ is the union of
$\det\overline\Lambda/\det\Lambda_i$ many cosets  modulo
$\overline\Lambda$, the union is certainly covered by
$\sum_{i=1}^s\frac{\det\overline\Lambda}{\det\Lambda_i}=m+s-1$ many cosets
of $\Lambda$ modulo $\overline\Lambda$. But here we have counted the
trivial coset at least $s$ times. ii) is a direct consequence of i).
\end{proof}
Lemma \ref{lem:lattice_cap} ii) implies, in particular,  that the  
 union of two strict sublattices can never be the whole lattice. This is
 no longer true for three sublattices, as we see in the next example, which also shows that
 Lemma \ref{lem:lattice_cap} ii) is not an equivalence. 
\begin{example}
Let $\Lambda=\Z^2$, and let $\Lambda_1,\ldots,\Lambda_4\subset\Z^2$ be
  the sublattices 
\begin{equation*}
\Lambda_1=\{(z_1,z_2)^\intercal\in\Z^2: z_2\equiv
  0\bmod 2\},\, \Lambda_2=\{(z_1,z_2)^\intercal\in\Z^2: z_1\equiv
  0\bmod 2\},
\end{equation*}
and
\begin{equation*}
\Lambda_3=\{(z_1,z_2)^\intercal\in\Z^2: z_2\equiv
  0 \bmod 3\},\, \Lambda_4=\{(z_1,z_2)^\intercal\in\Z^2: z_1\equiv
  z_2 \bmod 2\}.
\end{equation*}

Then $\Lambda_1\cup\Lambda_2\cup \Lambda_4=\Lambda$ but $\Lambda_1\cup\Lambda_2\cup \Lambda_3\neq\Lambda$.
Furthermore $\det\Lambda=1$, $\det\Lambda_1=\Lambda_2=\Lambda_4=2$, $\det\Lambda_4=3$
and 
\begin{equation*}
  \overline{\Lambda}=\Lambda_1\cap\Lambda_2\cap\Lambda_3=\{(z_1,z_2)^\intercal\in\Z^2: z_1\equiv
  0\bmod 2,z_2\equiv 0\bmod 6\}
\end{equation*}
with $\det\overline{\Lambda}=12$, while $ \sum_{i=1}^3
\frac{\det\overline{\Lambda}}{\det\Lambda_i} -1=14$.
\end{example}

We now come to the proof of the full-dimensional case.
\begin{proof}[Proof of Theorem \ref{thm:full}] Let
  $\Lambda_1,\dots,\Lambda_s$ be the  full-dimensional
  forbidden sublattices of the given lattice $\Lambda$ and let
  $\overline{\Lambda}=\cap_{i=1}^s\Lambda _i$. Let 
 \begin{equation*}
         m=\min\left\{\sum_{i=1}^s\frac{\det\overline\Lambda}{\det\Lambda_i}-s+1, \frac{\det\overline\Lambda}{\det\Lambda}\right\}.
 \end{equation*}
We first claim 
\begin{claim} Let $\lambda>0$ with 
  $\vol_T((\lambda\,\frac{1}{2}K)/\overline\Lambda)\geq m\det\Lambda$. Then 
\begin{equation*}
\lambda_1(K,\Lambda\setminus\cup_{i=1}^s\Lambda_i)\leq\lambda.
\label{eq:upper_bound_full}
\end{equation*}
\label{claim:one}
\end{claim}
In order to verify the claim,  we first assume
$m=\sum_{i=1}^s\frac{\det\overline\Lambda}{\det\Lambda_i}-s+1<\det\overline\Lambda/\det\Lambda$. 
By  Lemma \ref{lem:vandercorputtorus} $\lambda\,K$ contains $m+1$ lattice
points of $\Lambda$ belonging to different cosets with respect to
$\overline\Lambda$. 
By Lemma \ref{lem:lattice_cap} i), $\cup_{i=1}^s\Lambda_i$ is covered by at
most $m$ cosets of $\Lambda$ modulo $\overline\Lambda$, and thus, $\lambda\,K$
contains a lattice point of $\Lambda\setminus\cup_{i=1}^s\Lambda_i$. 

Next suppose that $m=\det\overline\Lambda/\det\Lambda$. Then
$\vol_T((\lambda\,\frac{1}{2}K)/\overline\Lambda)=\det\overline\Lambda=\vol_T(\R^n/\overline{\Lambda})$
and, in particular, $\lambda\,K$ contains a representative  
 of each coset of
$\Lambda$ modulo $\overline{\Lambda}$. By assumption there exists a
coset containing a point $a\in\Lambda\setminus\cup_{i=1}^s\Lambda_i$ and
hence, all points of this coset, i.e., $a+\overline{\Lambda}$, 
lie in $\Lambda\setminus\cup_{i=1}^s\Lambda_i$.

This verifies the claim and it remains to compute a $\lambda$
with 
\begin{equation}
\vol_T((\lambda\,\tfrac{1}{2}K)/\overline\Lambda)\geq m\det\Lambda.
\label{eq:to_show}
\end{equation}
 To this end we set
$\lambda_1=\lambda_1(K,\overline\Lambda)$ and we write an arbitrary
$\lambda>0$ modulo $\lambda_1$ in the form 
$\lambda=\left\lfloor\frac{\lambda}{\lambda_1}\right\rfloor \lambda_1 +
\rho\,\lambda_1$, 
with $0\leq \rho<1$.  Hence, in view of the sum theorem of Kneser and
Macbeath \eqref{eq:torus_add} and the packing property \eqref{eq:torus_scal} of $\lambda_1$
with respect to $\frac{1}{2}K$, we may write 
\begin{equation}
\begin{split}
 \vol_T\left((\lambda\,\frac{1}{2}K)/\overline\Lambda\right)&=
\vol_T\left(\left(\left(\left\lfloor\frac{\lambda}{\lambda_1}\right\rfloor
  \frac{\lambda_1}{2} + \rho\,\frac{\lambda_1}{2}\right)
\,K\right)/\overline\Lambda\right)\\
&\geq \vol_T\left(\underbrace{\left(\frac{\lambda_1}{2}K\right)/\overline\Lambda\oplus\cdots
  \oplus
  \left(\frac{\lambda_1}{2}K\right)/\overline\Lambda}_{\lfloor\lambda/\lambda_1\rfloor}\oplus
\left(\frac{\rho\lambda_1}{2}K\right)/\overline\Lambda\right)\\
&\geq
\min\left\{\left(\left\lfloor\frac{\lambda}{\lambda_1}\right\rfloor+\rho^n\right)\left(\frac{\lambda_1}{2}\right)^n\vol(K),
  \det\overline\Lambda\right\}.
\end{split}
\label{eq:chain_torus}
\end{equation}
Thus, \eqref{eq:to_show} is certainly satisfied for a $\overline\lambda$ with  
\begin{equation*}
\left(\left\lfloor\frac{\overline\lambda}{\lambda_1}\right\rfloor+\rho^n\right)\left(\frac{\lambda_1}{2}\right)^n\vol(K)
=\left(\sum_{i=1}^s\frac{\det\overline\Lambda}{\det\Lambda_i}-s+1\right)\det\Lambda. 
\end{equation*} 
Using 
\begin{equation}
\left\lfloor\frac{\lambda}{\lambda_1}\right\rfloor+\rho^n > 
\frac{\lambda-\lambda_1}{\lambda_1}
\label{eq:ineq}
\end{equation} 
we find  $\lambda_1(K,\Lambda\setminus\cup_{i=1}^s\Lambda_i)\leq\overline\lambda<\frac{2^n\det\Lambda}{\lambda_1^{n-1}\vol(K)}
\biggl(\sum_{i=1}^s \frac{\det\overline{\Lambda}}{\det\Lambda_i} -s+1
\biggr)+
\lambda_1$.
\end{proof}

\begin{remark} The bound in Theorem \ref{thm:full} can slightly be
  improved in lower dimensions by noticing that in \eqref{eq:ineq}
  we may replace $\frac{\lambda-\lambda_1}{\lambda_1}$
  by  
\begin{equation*}
\frac{\lambda}{\lambda_1}-\rho+\rho^n.
\end{equation*} 
Since $\rho-\rho^n$ takes its maximum at $\rho=(1/n)^{1/(n-1)}$ we get in
this way 
\begin{equation*}
\begin{split}
\lambda_1\left(K,
  \Lambda\setminus\cup_{i=1}^s\Lambda_i\right)& \leq \frac{2^n\det\Lambda}{\lambda_1(K,\overline{\Lambda})^{n-1}\vol(K)}
\biggl(\sum_{i=1}^s \frac{\det\overline{\Lambda}}{\det\Lambda_i} -s+1
\biggr)\\ &+
n^{\frac{-1}{n-1}}\left(\frac{n^n-1}{n^n}\right)\lambda_1(K,\overline{\Lambda}).
\end{split}
\end{equation*}
\end{remark}
Since we certainly have 
\begin{equation*}
\lambda_i\left(K, \Lambda\setminus\cup_{i=1}^s\Lambda_i\right)\leq
\lambda_1\left(K,
  \Lambda\setminus\cup_{i=1}^s\Lambda_i\right)  + \lambda_{i-1}(K,\overline{\Lambda}),
\end{equation*}
we have the following straightforward extension of Theorem \ref{thm:full} to
higher successive minima 
\begin{corollary} Under the assumptions of Theorem \ref{thm:full} we
  have for $1\leq i\leq n$. 
\begin{equation*}
\begin{split}
\lambda_i\left(K,
  \Lambda\setminus\cup_{i=1}^s\Lambda_i\right)&\leq \frac{2^n\det\Lambda}{\lambda_1(K,\overline{\Lambda})^{n-1}\vol(K)}
\biggl(\sum_{i=1}^s \frac{\det\overline{\Lambda}}{\det\Lambda_i} -s+1
\biggr)\\ &+
\lambda_1(K,\overline{\Lambda})+\lambda_{i-1}(K,\overline{\Lambda}),
\end{split}
\end{equation*}
where we set $\lambda_{0}(K,\overline{\Lambda})=0$.
\end{corollary}

In the case $s=1$ one can state a slightly better  bound 
\begin{corollary} Under the assumptions of Corollary \ref{cor:onefull} we
  have for $1\leq i\leq n$. 
\begin{equation*}
\begin{split}
\lambda_i\left(K,
  \Lambda\setminus\Lambda_1\right)&\leq
\frac{2^n\det\Lambda}{\lambda_1(K,\Lambda_1)^{n-1}\vol(K)} +
\lambda_1(K,\Lambda)+\lambda_{i}(K,\Lambda). 
\end{split}
\end{equation*}
\end{corollary}
\begin{proof} In view of Corollary \ref{cor:onefull} it
  suffices to show  $\lambda_i\left(K,
  \Lambda\setminus\Lambda_1\right)\leq  
\lambda_1(K,\Lambda\setminus\Lambda_1)+\lambda_{i}(K,\Lambda)$ for
$i=2,\dots,n$. To this end let $a\in
\lambda_1(K,\Lambda\setminus\Lambda_1)\,K\cap\Lambda$ and let
$b_1,\dots,b_n$ be linearly independent with $b_j\in
\lambda_j(K,\Lambda)\,K\cap\Lambda$, $j=1,\dots,n$. Since not both, $b_j$ and 
 $a+b_j$, can belong to the forbidden sublattice $\Lambda_1$ we can
 select from each pair $b_j$, $a+b_j$ one contained in
 $\Lambda\setminus\Lambda_1$, $1\leq j\leq n$. Let these points
 be denoted by $\overline{b}_j$, $j=1,\dots,n$. Then
 $a,\overline{b}_j\in\left(\lambda_1(K,\Lambda\setminus\Lambda_1)+\lambda_{j}(K,\Lambda)\right)K$,
           $1\leq j\leq n$.
 
Now choose $k$ such that $a\notin\lin(\{b_1,\dots,b_n\}\setminus\{b_k\})$.
Then the lattice points
$a,\overline{b}_1,\dots\overline{b}_{k-1},
\overline{b}_{k+1},\dots,\overline{b}_n$ are linearly independent and
we are done. 
\end{proof}

\begin{remark}
It is also possible to extend lower-dimensional lattices to lattices of full rank by
adjoining ``sufficiently large'' vectors, i.e. for each $\Lambda_i$ of rank
$n_i$ choose linearly independent $z_{i,n_i+1},\ldots,z_{i,n}\in
\Lambda\setminus\Lambda_i$ and consider the lattice $\overline{\Lambda}_i$ spanned
by $\Lambda_i$ and $z_{i,n_i+1},\ldots,z_{i,n}$.
If $z_{i,j}$ are such that $\lambda_j(K,\overline{\Lambda}_i)$ is very large for
$j>n_i$,
one can apply the results from Section~\ref{sec:full} to the collection
$\overline{\Lambda}_i$, $1\leq i\leq s$.
However, the bounds obtained in this way are in general weaker, with one exception in the case $s=1$
for the bound on $\lambda_1(K\cap\Lambda\setminus \Lambda_1)$.
Here we get   
\begin{equation*}
 \lambda_1\left(K,
  \Lambda\setminus\Lambda_1\right)\leq \frac{2^n\det\Lambda}{\lambda_1(K,\Lambda_1)^{n-1}\vol(K)}+
\lambda_1(K,\Lambda)
\end{equation*}
for $\Lambda_1\subsetneq\Lambda$ with $\rg \Lambda_1<n$, which
improves on Theorem~\ref{thm:lower}. 
\end{remark}

\end{document}